\documentclass{article}

\usepackage{latexsym}
\usepackage{a4wide}
\usepackage{amscd}
\usepackage{graphics}
\usepackage{amsmath}
\usepackage{amssymb}
\usepackage{mathrsfs}
\input xy
\xyoption{all}

\newcommand{\N}{\mathbb{N}}                     % the natural numbers
\newcommand{\Z}{\mathbb{Z}}                     % the integer numbers
\newcommand{\R}{\mathbb{R}}                     % the real line
\newcommand{\T}{\mathbb{T}}                     % the torus
\newcommand{\set}[2]{\left\{{#1}\mid{#2}\right\}}       % the set
\newcommand{\qed}{\hfill $\Box$ \bigskip}       % end of proof
\newcommand{\supp}{\mathrm{supp\,}}             % support
\newcommand{\graf}{\mathrm{graph\,}}            % graph
\newcommand{\grad}{\mathrm{grad\,}}		% gradient
\newcommand{\hess}{\mathrm{Hess\,}}		% Hessian
\newcommand{\sign}{\mathrm{sign\,}}		% signature
\newcommand{\scal}[1]{\langle{#1}\rangle}     % scalar product

\newtheorem{thm}{\sc Theorem}[section]      % numbered within each section
\newtheorem{cor}[thm]{\sc Corollary}        % numbered along with Theorem
\newtheorem{lem}[thm]{\sc Lemma}            % numbered along with Theorem
\newtheorem{prop}[thm]{\sc Proposition}     % numbered along with Theorem
\newtheorem{rem}[thm]{\sc Remark}	    % numbered along with Theorem  
\title{Invariant measures of Hamiltonian systems with prescribed
  asymptotic Maslov index}

\author{Alberto Abbondandolo and Alessio Figalli}
\date{September 29, 2007}

\begin{document}

\maketitle

\renewcommand{\theenumi}{\roman{enumi}}
\renewcommand{\labelenumi}{(\theenumi)}

\centerline{\em Dedicated to Vladimir Igorevich Arnold}

\begin{abstract}
We study the properties of the asymptotic Maslov index of invariant
measures for time-periodic Hamiltonian systems on the cotangent bundle 
of a compact manifold $M$. We show that if $M$ has finite fundamental 
group and the Hamiltonian satisfies some general growth assumptions on 
the momenta, the asymptotic Maslov indices of periodic orbits are
dense in the half line $[0,+\infty)$. Furthermore, if the
Hamiltonian is the Fenchel dual of an electro-magnetic Lagrangian, every
non-negative number $r$ is the limit of the asymptotic Maslov indices of a
sequence of periodic orbits which converges narrowly
to an invariant measure with asymptotic Maslov index $r$. We discuss
the existence of minimal ergodic invariant measures with prescribed asymptotic
Maslov index by the analogue of Mather's theory of the beta function, the 
asymptotic Maslov index playing the role of the rotation vector.
\end{abstract}

\section*{Introduction}

The Lagrangian Grassmannian $\mathscr{L}(n)$ is the space of $n$-dimensional linear subspaces of $\R^{2n}$ on which the
symplectic form $\omega_0 = \sum_{j=1}^n dp^j \wedge dq^j$ vanishes. The fundamental group of  $\mathscr{L}(n)$ is infinite cyclic. The Maslov index of a closed loop of Lagrangian
subspaces is a measure of its winding number in $\mathscr{L}(n)$.
It was introduced by Maslov in his book on perturbation methods in
quantum mechanics \cite{mas72}, and its geometric interpretation both
as a topological winding number and as an algebraic intersection number was
discussed by Arnold in an appendix of the same book and in \cite{arn67}.
Here we will use the definition by Robbin and Salamon \cite{rs93} of
the Maslov index $\mu(\lambda,\lambda_0)$ of a path $\lambda$ in
$\mathscr{L}(n)$ with arbitrary end-points with respect to some
$\lambda_0\in \mathscr{L}(n)$.

If $x(t)=(q(t),p(t))$ is the orbit of a point $(q_0,p_0)$ by
a Hamiltonian flow on the phase space $\R^{2n}$, we can define
$\mu_T(q_0,p_0)$ as the Maslov index of the
evolution of the vertical space $(0) \times \R^n$ by the differential
of the flow at $(q_0,p_0)$, with respect to the vertical space itself, over the time
interval $[0,T]$. When the Hamiltonian is strictly convex in the
momentum variables $p$, this Maslov index is non-negative, and it
coincides - up to a suitable additive constant - with the Morse index
of the curve $q(t)$, seen as an extremal curve of the Lagrangian
action functional with fixed end-points at $t=0$ and $t=T$,
as proved by Duistermaat in \cite{dui76}.
For a general Hamiltonian, this interpretation as a Morse index is
recovered by Floer's approach to the study of the Hamiltonian action
functional, as shown by Salamon and Zehnder in \cite{sz92}.
These facts remain true if we replace the phase space
$\R^{2n}$ by the cotangent bundle $T^*M$ of an arbitrary manifold, and
the vertical Lagrangian space $(0)\times \R^n$ by the vertical
subbundle $T^v T^*M = \ker D\pi$, where $\pi:T^*M \rightarrow M$
denotes the projection.

In this paper we are interested in the {\em asymptotic Maslov index},
that is in the limit
\[
\hat{\mu}(q_0,p_0) := \lim_{T \rightarrow +\infty} \frac{\mu_T(q_0,p_0)}{T}.
\]
Ergodic theory implies that the above limit exists for almost every
initial condition $(q_0,p_0)$, with respect to any invariant
probability measure $\eta$ on $T^*M$, and that the {\em asymptotic
  Maslov index} of an invariant probability measure $\eta$ is well-defined:
\[
\hat{\mu}(\eta) := \int \hat{\mu} \, d\eta = \lim_{T \rightarrow
  +\infty} \int \frac{\mu_T}{T} \, d\eta,
\]
provided that the functions $\mu_T$ are in $L^1(T^*M,\eta)$. This
invariant was introduced by Ruelle \cite{rue85}, who also showed that
the asymptotic Maslov index of an invariant measure is continuous with
respect to the narrow topology of measures and with respect to
perturbations of the Hamiltonian system, a fact which fails to be true
for other ergodic invariants such as the Lyapunov exponents. Under
different names, this invariant appears in the context of systems
defined by a Hamiltonian which is convex in the $(q,p)$ variables in
the book of Ekeland \cite{eke90}, and in the context
of Lagrangian systems in \cite{abb94} and \cite{ba96}. Contreras, Gambaudo, Iturriaga and Paternain \cite{cgip03} have
extended this invariant to symplectic manifolds with a
fixed Lagrangian bundle, and have studied the properties of the asymptotic
Maslov index of the Liouville measure on the energy level of
optical Hamiltonians.

Here we deal mainly with time-periodic Hamiltonian systems on the
cotangent bundle of a compact manifold $M$. We start by recalling the
definition and the continuity properties of the asymptotic
Maslov index in this context, paying special attention to invariant
measures which are not compactly supported (as often happens when
there is no conservation of energy). Let $I(H)\subset \R$ be the
set - actually an interval -  consisting of the asymptotic
Maslov indices of the invariant measures.
Our first result is that when $M$ is compact and has finite
fundamental group and the Hamiltonian satisfies some very general growth
assumptions, the interval $I(H)$ contains the whole half-line
$[0,+\infty)$. Actually, the asymptotic Maslov indices of the
contractible closed orbits with integer period are dense in
$[0,+\infty)$. The proof is based on Floer's approach by
$J$-holomorphic curves, together with a result from \cite{af07}
which allows to deal with general growth conditions on the Hamiltonian.
A large interval $I(H)$ is in sharp contrast with what may happen with 
Hamiltonian systems on configuration spaces having infinite
fundamental group: all the invariant measures for the geodesic flow
on a compact Riemannian manifold with non-positive sectional curvature - which
necessarily has infinite fundamental group - have vanishing
asymptotic Maslov index, so $I(H)=\{0\}$ in this case.

If a positive number $r$ is not the asymptotic Maslov index
of any contractible periodic orbit - generically we expect a
countable number of periodic orbits - it seems natural to consider
a sequence of periodic orbits with asymptotic Maslov index
converging to $r$ and try to take a limit. Our second and
main result is that this is indeed possible for a Hamiltonian
which is the Fenchel transform of an electro-magnetic Lagrangian
\[
L(t,q,v) = \frac{1}{2} \langle v,v \rangle + \langle A(t,q),v \rangle
- V(t,q),
\]
provided that the manifold $M$ is compact and has finite fundamental group.
Under these assumptions, we prove that for any $r \geq 0$ at least one of the following
statements must hold:
\begin{enumerate}
\item There is a contractible closed orbit with integer period and
  asymptotic Maslov index $r$.

\item There is a sequence of contractible closed orbits with
  integer minimal periods
  $k_j \rightarrow +\infty$ and asymptotic Maslov indices $\hat{\mu}_j
  \rightarrow r$. The corresponding sequence of probability
  measures narrowly converges to an invariant probability
  measure $\eta$ with finite second
  moment and asymptotic Maslov index $\hat{\mu}(\eta)=r$.
\end{enumerate}
In the last section we outline the analogue of Mather's theory of
the beta function, the asymptotic Maslov index $\hat{\mu}\in I(H)$
playing the role of the rotation class $\rho\in H_1(M,\R)$. For
sake of simplicity, we deal with electromagnetic Lagrangians on
compact manifolds with finite fundamental group, so that
$I(H)=[0,+\infty)$ and for any $r\geq 0$ there is an invariant
measure with asymptotic Maslov index $r$ and finite action. In
this case, the function
\[
\beta : [0,+\infty) \rightarrow \R,
\]
which associates to $r$ the minimal Lagrangian action over all
the invariant measures with asymptotic Maslov index $r$ is
convex and has quadratic growth. In
particular, $\beta$ must have infinitely many strict extremal points,
and if $r$ is such a point the minimum $\beta(r)$ is
achieved by an ergodic invariant measure.
Notice that the topological assumption on $M$ implies that $H_1(M,\R)=(0)$, so
Mather theory would produce just one class of minimal invariant measures, while
the fact that $I(H)=[0,+\infty)$ allows to find 
a diverging sequence of numbers 
which are asymptotic Maslov indices of ergodic action-minimizing measures.
Of course, in this situation there are also infinitely
many periodic orbits, which produce ergodic invariant measures, but
these measures need not minimize the action
among invariant measures with prescribed asymptotic Maslov index.

Other related problems remain open. Mather \cite{mat91} has proved that
measures minimizing the action with given rotation class are
supported in the image of a Lipschitz section of $TM$. What is the
structure of a minimizing measure with given asymptotic Maslov index?
Ma\~n\'e \cite{man96} has proved that a measure which minimizes the
action among {\em closed} measures is invariant. Can one extend the
notion of asymptotic Maslov index to some class of closed measures,
and prove that a measure which minimizes the action
among closed ones with given asymptotic Maslov index is invariant?

\paragraph{Acknowledgment.} We wish to express our gratitude to 
Felix Schlenk for introducing us to Gromov's length estimate.
 
\section{The Maslov index}

In this section we review the basic properties of the Maslov index,
following \cite{rs93}.

We consider coordinates $(q^1,\dots, q^n,p^1,\dots, p^n)$ on the
Euclidean space $\R^{2n}$, endowed with its standard symplectic
structure, given by the alternating 2-form
\[
\omega_0 = \sum_{j=1}^n dp^j \wedge dq^j.
\]
The symplectic group - that is the
group of linear isomorphisms of $\R^{2n}$ preserving $\omega_0$ - is
denoted by $\mathrm{Sp}(2n)$. Let $\mathscr{L}(n)$ be the manifold of
Lagrangian subspaces of $(\R^{2n},\omega_0)$, that is $n$-dimensional
linear subspaces of $\R^{2n}$ on which $\omega_0$ vanishes
identically.

Let $\lambda:[a,b] \rightarrow \mathscr{L}(n)$
be a continuous path of Lagrangian subspaces, and let $\lambda_0$ be a
fixed Lagrangian subspace. The {\em Maslov index}
$\mu(\lambda,\lambda_0)$ is a half-integer counting the
intersections of $\lambda(t)$ with $\lambda_0$, with suitable
orientation signs and multiplicity. Rather than defining it, we prefer
to list its characterizing properties:
\begin{description}
\item[Naturality.] If $\Phi\in \mathrm{Sp}(2n)$, then $\mu(\Phi
  \lambda,\Phi \lambda_0) =
  \mu(\lambda,\lambda_0)$.
\item[Juxtaposition.] If $a<c<b$, then $\mu(\lambda,\lambda_0) =
\mu(\lambda|_{[a,c]},\lambda_0) +\mu(\lambda|_{[c,b]},\lambda_0)$.
\item[Product.] If $n'+n''=n$ and $\mathscr{L}(n') \times
  \mathscr{L}(n'')$ is identified with a submanifold of
  $\mathscr{L}(n)$ in the obvious way, then $\mu(\lambda'\oplus
  \lambda'',\lambda_0' \oplus \lambda''_0) = \mu(\lambda',\lambda_0')
+  \mu(\lambda'',\lambda_0'')$.
\item[Homotopy.] If $\lambda,\lambda': [a,b] \rightarrow
  \mathscr{L}(n)$ have the same end-points and are homotopic with
  fixed end-points, then $\mu(\lambda,\lambda_0) =
  \mu(\lambda',\lambda_0)$.
\item[Localization.] If $\lambda(t)=\graf S(t)$, the graph of a path of
  symmetric endomorphisms of $\R^n$, 
  then $\mu(\lambda,\R^n \times (0))= (\sign S(b)-\sign S(a))/2$, where
  $\sign$ denotes the signature.
\item[Zero.] For every path $\lambda: [a,b] \rightarrow
  \mathscr{L}(n)$ such that the dimension of $\lambda(t)\cap
  \lambda_0$ does not depend on $t$, the Maslov index
  $\mu(\lambda,\lambda_0)$ is zero.
\end{description}

A related invariant is the {\em Conley-Zehnder index} $\mu_{CZ}(\Phi)$
of a continuous path $\Phi:[a,b] \rightarrow \mathrm{Sp}(2n)$ in the symplectic
group. It can be defined as
\[
\mu_{CZ}(\Phi) = \mu(\graf \Phi, \Delta),
\]
the Maslov index of the graph of $\Phi$ with respect to the diagonal
$\Delta$ of $\R^{2n}\times \R^{2n}$. Here $\R^{2n} \times \R^{2n}$ is
endowed with the symplectic form $\omega \oplus (-\omega)$, so that the diagonal, and more generally the graph of every
symplectic isomorphism, is a Lagrangian subspace. When $\Phi(a)=I$
and $\Phi(b)-I$ is invertible, $\mu_{CZ}(\Phi)$ is an integer. The
Conley-Zehnder index is significant when $\Phi: \R
\rightarrow \mathrm{Sp}(2n)$ is the fundamental
solution of a linear $k$-periodic Hamiltonian system. In this
situation, the limit
\[
\hat{\mu}_{CZ}(\Phi):=
\lim_{h\rightarrow +\infty} \frac{\mu_{CZ}(\Phi|_{[0,hk]})}{hk}
\]
exists, and
\begin{equation}
\label{bott}
|\mu_{CZ}(\Phi|_{[0,hk]}) - hk \, \hat{\mu}_{CZ}(\Phi)| \leq 2n,
\end{equation}
as it is proved in \cite{sz92}. 
See also \cite{lon02} for finer estimates on the behavior
of the Conley-Zehnder index under iteration.
The quantity $\hat{\mu}_{CZ}(\Phi)$ is called the {\em Bott index}
of the linear Hamiltonian system solved by $\Phi$.

Let $\Phi:[a,b] \rightarrow \mathrm{Sp}(2n)$ be a continuous path in the
symplectic group, and let $\lambda_0\in \mathscr{L}(n)$. Then
\begin{equation}
\label{czvmas}
|\mu_{CZ}(\Phi) - \mu(\Phi\lambda_0,\lambda_0)| \leq 2n.
\end{equation}
In fact, by Theorem 3.2 in \cite{rs93},
\[
\mu(\Phi \lambda_0,\lambda_0) = \mu(\graf \Phi, \lambda_0\times \lambda_0),
\]
so
\[
\mu_{CZ}(\Phi) - \mu(\Phi\lambda_0,\lambda_0) = \mu (\graf \Phi, \Delta) - \mu
(\graf \Phi, \lambda_0\times \lambda_0)
\]
is the {\em H\"ormander index} of the 4-uple of Lagrangian subspaces
of $\R^{2n} \times \R^{2n}$
\[
(\lambda_0\times \lambda_0, \Delta; \graf \Phi(a), \graf \Phi(b)),
\]
a half-integer which does not depend on the path $\Phi$ and whose absolute value does not exceed $2n$ (see \cite{hor71}, \cite{dui76}, or
\cite[Theorem 3.5]{rs93}).

\section{The asymptotic Maslov index on cotangent bundles}
\label{sect:asympMaslov}

Let $M$ be a compact $n$-dimensional smooth manifold without
boundary. The manifold $T^*M$ is endowed with the symplectic structure
\[
\omega = \sum_{j=1}^n dp^j \wedge dq^j,
\]
where $(q^1,\dots,q^n,p^1,\dots ,p^n)$ are cotangent local coordinates
on $T^*M$. The fibers of the vertical subbundle of $TT^*M$,
\[
T^v_x T^*M = \ker D \pi(x),
\]
where $\pi:T^*M \rightarrow M$ is the projection, are Lagrangian
subspaces of the symplectic vector spaces $T_x T^*M$.

Let $H\in C^{\infty}(\T \times
T^*M)$ be a time-periodic Hamiltonian on the cotangent bundle of $M$.
Here $\T=\R / \Z$ indicates the circle. The induced Hamiltonian vector
field $X_H$ is the time-periodic tangent vector field on $T^*M$
defined by
\[
\omega(X_H(t,x),\xi) = - D_x H(t,x)[\xi], \quad \forall \xi\in T_x
T^*M.
\]
At the moment, we only assume that $X_H$ generates a global flow on
the extended phase space $\T \times T^*M$, that is
\[
\phi: \R \times \T \times T^*M \rightarrow \T \times T^*M, \quad
\phi_t(s,x_0)=(s+t,\varphi_{s+t}(s,x_0)),
\]
where $\varphi: \R \times \T \times T^*M \to T^*M$ solves the Cauchy problem
\[
\partial_t \varphi_t(s,x_0) = X_H(t,\varphi_t(s,x_0)), \quad
\varphi_s(s,x_0)=x_0.
\]
Since $M$ is compact, this is true for instance if $H$ is coercive and
it satisfies the estimate
\[
\partial_t H(t,x) \leq c (1 + H(t,x)), \quad \forall (t,x) \in \T
\times T^*M,
\]
for some positive constant $c$.

Consider the orbit $x(t):=\varphi_t(s,x_0)$ such that $x(s)=x_0\in T^*M$.
The isomorphism
\[
D_x\varphi_t(s,x_0) : T_{x_0} T^*M \rightarrow T_{x(t)} T^*M
\]
is symplectic, so
\[
\lambda_t(s,x_0) := D_x\varphi_t(s,x_0) T_{x_0}^v T^*M
\]
is a Lagrangian subspace of $T_{x(t)}T^*M$. By use of a symplectic
trivialization of $x^*(TT^*M)$,
\[
U(t) : (\R^{2n},\omega_0) \longrightarrow (T_{x(t)} T^*M, \omega_{x(t)})
\]
mapping the Lagrangian subspace $\lambda_0 := (0)\times \R^n$ into the
vertical subspace $T_{x(t)}^v T^*M$, we can transform the path
$t\mapsto \lambda_t(s,x_0)$ into a path
\[
\lambda^U(t) := U(t)^{-1} \lambda_t(s,x_0)
\]
of Lagrangian subspaces of $(\R^{2n},\omega_0)$. By the Juxtaposition, Homotopy, and Zero properties of the Maslov index, the number
\[
\mu_t(s,x_0) := \mu(\lambda^U|_{[s,s+t]},\lambda_0) - \frac{n}{2}
\]
does not depend on the choice of the trivialization $U$ (see for
instance \cite{oh97, as06}). The normalization constant $-n/2$ appears
because when the Hamiltonian is strictly convex on the fibers and
$t$ is positive but smaller than the first conjugate instant, we would like
$\mu_t(s,x_0)$ to be zero.
The function $\mu_t$ is neither lower nor
upper semicontinuous, but if we denote by $\underline{\mu}_t$ and
$\overline{\mu}_t$  its lower and upper semi-continuous
envelopes, we have the bounds
\begin{equation}
\label{squeeze}
\underline{\mu}_t \leq \mu_t \leq \overline{\mu}_t, \quad \overline{\mu}_t
- \underline{\mu}_t \leq 2n.
\end{equation}

Now we would like to consider the limit for $t\to +\infty$ of the
quotient $\mu_t/t$. This is possible thanks to the following:

\begin{thm}
\label{erg}
Let $\eta$ be a Borel probability measure on $\T \times T^*M$ such
that:
\begin{enumerate}
\item $\eta$ is invariant for the flow $\phi$;
\item for every $t\geq 0$ the function $\mu_t$ is in $L^1(\T \times
T^*M,\eta)$.
\end{enumerate}
Then for $\eta$-almost every  $(s,x_0)\in \T \times T^*M$ the limit
\[
\hat{\mu}(s,x_0) := \lim_{t\rightarrow +\infty} \frac{\mu_t(s,x_0)}{t}
\]
exists, and it defines a $\phi$-invariant function in $L^1(\T \times
T^*M,\eta)$. Furthermore,
\[
\int_{A} \hat{\mu} (s,x_0) \, d\eta(s,x_0) =
\lim_{t\rightarrow +\infty} \frac{1}{t} \int_{A}
\mu_t(s,x_0) \, d\eta(s,x_0),
\]
for every $\phi$-invariant Borel set $A\subset \T \times T^*M$.
\end{thm}

The quantity $\hat{\mu}(s,x_0)$ is the {\em asymptotic Maslov index} of
the point $(s,x_0)$, and the quantity
\[
\hat{\mu}(\eta) := \int_{\T \times T^*M} \hat{\mu} (s,x_0) \, d\eta(s,x_0) =
\lim_{t\rightarrow +\infty} \frac{1}{t} \int_{\T \times T^*M}
\mu_t(s,x_0) \, d\eta(s,x_0)
\]
is the {\em asymptotic Maslov index} of the invariant measure
$\eta$. The function $\mu$ is bounded on compact subsets of $\R \times
\T \times T^*M$, so the summability assumption (ii) is trivially
satisfied when $\eta$ has compact support.

Theorem \ref{erg} is an easy consequence of Kingman subadditive ergodic
theorem. In Derriennic's formulation \cite{der75}, this theorem
says that if $\phi_t$  is a measure preserving flow on the
probability space $(\Omega,\mathcal{F},\eta)$ and $f_t$ is
one-parameter family of real valued functions in
$L^1(\Omega,\mathcal{F},\eta)$ satisfying
\begin{equation}
\label{hyp}
\int_{\Omega} f_t(\omega) \, d\eta (\omega) \geq -c(t+1), \quad
f_{s+t} \leq f_s + f_t \circ \phi_s, \quad \forall t,s\in [0,+\infty),
\end{equation}
for some positive constant $c$,
then the limit of $f_t/t$ exists $\eta$-almost everywhere, and the
integral of such a limit over any $\phi$-invariant
$A\in \mathcal{F}$ coincides with the limit of the integral of
$f_t/t$ over $A$.
The Naturality and the Juxtaposition properties of the Maslov index imply that
\begin{equation}
\label{subadd}
|\mu_{t+t'}(s,x) - \mu_t(s,x) - \mu_{t'}(\phi_t(s,x))| \leq 2n,
\end{equation}
It easily follows that, if $\Omega:=\T \times T^*M$ and $\mathcal{F}$
is its Borel sigma-algebra, the one-parameter family of functions
\[
f_t(s,x) := \mu_t(s,x) + 2n
\]
satisfies the hypotheses (\ref{hyp}) of the subadditive ergodic
theorem. It is also possible to derive Theorem \ref{erg}
from the standard Birkhoff ergodic theorem, by lifting the dynamical
system $\phi$ to the Lagrangian bundle of $T^*M$, as in
\cite{cgip03}. By means of such a lift, the asymptotic Maslov index
can be seen as a special case of Schwartzman asymptotic cycles, see
\cite{sch57}.

In the case of an autonomous Hamiltonian, one can work with the phase
space $T^*M$ instead of $\T \times T^*M$, and it is natural to
apply the analogue of Theorem \ref{erg} to the invariant measure
induced by the Liouville measure on each energy surface.

In the case of a time-periodic Hamiltonian, it is better to think of
an invariant measure as a generalized orbit. Let us examine the case
of an invariant measure associated to a $k$-periodic orbit, with $k$ a
positive integer.
We can identify a $k$-periodic orbit $x:\R/k\Z \rightarrow T^*M$ of
$X_H$ with the invariant probability measure $\eta$ on $\T \times
T^*M$ defined by
\[
\eta (A) = \frac{1}{k} |\set{t\in [0,k]}{(t,x(t))\in A}|,
\]
for every Borel subset of $\T \times T^*M$, where $|\cdot|$ denotes
the Lebesgue measure on $\R$. We recall that the $k$-periodic orbit
$x$ has a well defined Conley-Zehnder index
\[
\mu_{CZ}^k(x):= \mu_{CZ}(\Phi),
\]
where $\Phi:[0,k] \rightarrow \mathrm{Sp}(2n)$ is the path in the
symplectic group
$\mathrm{Sp}(2n)$ obtained by conjugating $t \mapsto D_x
\varphi_t(0,x(0))$ by a symplectic trivialization preserving the
vertical subbundle (see for instance \cite{sz92}, \cite{web02}, or
\cite{as06}). By (\ref{czvmas}),
\[
|\mu_{CZ}^k(x) - \mu_k(0,x(0))| = |\mu_{CZ}(\Phi) -
 \mu(\Phi\lambda_0,\lambda_0)| \leq 2n.
\]
From this estimate we deduce that the asymptotic Maslov index of the
invariant probability measure $\eta$ coincides with the {\em Bott
  index} of the $k$-periodic orbit $x$,
\[
\hat{\mu} (\eta) = \lim_{h\rightarrow +\infty}
\frac{\mu^{hk}_{CZ}(x)}{hk}.
\]
Then the inequality (\ref{bott}) with $h=1$ yields the estimate
\begin{equation}
\label{bott2}
|\mu_{CZ}^k(x) - k \, \hat{\mu}(\eta)| \leq 2n.
\end{equation}

Other interpretations of the asymptotic Maslov index are possible. For
instance, if one fixes a metric on $M$, one can consider the polar
decomposition for the differential of the flow. The orthogonal part in
this decomposition is actually unitary with respect to the almost
complex structure associated to the metric, and one can recover the
asymptotic Maslov index by taking the asymptotic time average of the
logarithm of the complex determinant of these unitary
isomorphisms. This is the original approach adopted by Ruelle in
\cite{rue85} when the phase space is $\R^{2n}$.

\section{The continuity of the asymptotic Maslov index}

We denote by $\mathscr{P}(\T\times T^*M)$ the space of Borel
probability measures on $\T \times T^*M$. We recall that a sequence
$(\eta_j)\in \mathscr{P}(\T \times T^*M)$ {\em narrowly} converges to
a measure $\eta\in \mathscr{P}(\T \times T^*M)$ if it converges to
$\eta$ in the weak-* topology of the dual space of $C_b(\T \times
T^*M)$, that is if
\[
\lim_{j\rightarrow +\infty}
\int_{\T \times T^*M} f\, d\eta_j = \int_{\T \times T^*M}
f\, d\eta,
\]
for every bounded continuous function $f$ on $\T\times T^*M$. The
narrow convergence is induced by a metrizable topology, so we do not
have to distinguish between continuity and sequential continuity. The
narrow limit of an invariant measure is invariant.

As observed by Ruelle \cite{rue85}, the asymptotic Maslov index is
continuous with respect to the narrow topology. We recall that a Borel
function $f$ is said to be uniformly integrable with respect to the
sequence of measures $(\eta_j)$ if
\[
\lim_{c\rightarrow +\infty} \sup_{j\in \N}\int_{\{|f|\geq c\}} |f| \, d\eta_j = 0.
\]

\begin{thm}
\label{cont}
Assume that $(\eta_j)\subset \mathscr{P}(\T \times T^*M)$ is a
sequence of $\phi$-invariant probability measures which narrowly
converges to $\eta$, and that for
every $t\geq 0$ the function $\mu_t$ is uniformly integrable with
respect to the sequence of measures $(\eta_j)$.
Then $(\hat{\mu}(\eta_j))$ converges to $\hat{\mu}(\eta)$.
\end{thm}

The following proof is adapted from the appendix in \cite{rue85}.
See \cite{cgip03} for a proof using the already mentioned lift to the
Lagrangian bundle of $T^*M$.

\begin{proof}
Set $M_t(\eta):= \int \mu_t \, d\eta$, so that
\begin{equation}
\label{limi}
\lim_{t\rightarrow +\infty} \frac{M_t(\eta)}{t} = \hat{\mu}(\eta).
\end{equation}
Integrating the inequality (\ref{subadd}) with
respect to a $\phi$-invariant probability measure $\eta$ we obtain
\[
|M_{t+t'}(\eta) - M_t(\eta) - M_{t'}(\eta)| \leq 2n, \quad \forall
 t,t'\in \R.
\]
It easily follows that
\[
|M_{ht}(\eta) - h M_t(\eta)| \leq 2(h-1)n, \quad \forall h\in \N, \;
 h\geq 1, \; \forall t\in \R.
\]
Then, if $t>0$,
\[
\left| \frac{M_{ht}(\eta)}{ht} - \frac{M_{t}(\eta)}{t} \right| \leq
\frac{2(h-1)n}{ht},
\]
and taking the limit for $h\rightarrow +\infty$, we obtain the estimate
\begin{equation}
\label{unif}
\left| \hat{\mu}(\eta) - \frac{M_{t}(\eta)}{t} \right| \leq
\frac{2n}{t},
\end{equation}
which shows that the limit (\ref{limi}) is uniform over the space of
all $\phi$-invariant probability measures. Fix a number $t\geq 0$.
By (\ref{squeeze}), the
uniform integrability of $\mu_t$ implies that of
$\underline{\mu}_t$. Since $\underline{\mu}_t$ is also
lower-semicontinuous, from the semicontinuity property of narrow limits
(see Lemma 5.1.7 in \cite{ags05}) we find
\[
\liminf_{j\rightarrow +\infty} M_t(\eta_j) \geq \liminf_{j\rightarrow
  +\infty} \int_{\T \times T^*M} \underline{\mu}_t \, d\eta_j \geq
  \int_{\T \times T^*M} \underline{\mu}_t \, d\eta \geq M_t(\eta) -
  2n.
\]
Similarly,
\[
\limsup_{j\rightarrow +\infty} M_t(\eta_j) \leq \limsup_{j\rightarrow
  +\infty} \int_{\T \times T^*M} \overline{\mu}_t \, d\eta_j \leq
  \int_{\T \times T^*M} \overline{\mu}_t \, d\eta \leq M_t(\eta) +
  2n.
\]
The above inequalities together with the uniform estimate
(\ref{unif}) imply that $(\hat{\mu}(\eta_j))$ converges to
$\hat{\mu}(\eta)$.
\end{proof} \qed

\section{The range of the asymptotic Maslov index}

We denote by $I(H)\subset \R$ the interval consisting of the
asymptotic Maslov indices assumed by all the invariant probability
Borel measures $\eta$ for which $\mu_t$ is integrable for every $t\geq 0$.
In general, a Hamiltonian vector field may have no invariant
probability measures, so $I(H)$ could be empty.

Choose as $H$ the energy of the geodesic flow on a compact 
Riemannian manifold $M$ with non-positive sectional curvature. 
Since every geodesic on such a
manifold has no conjugate points, the Maslov index of every orbit
of the corresponding Hamiltonian system is zero. Therefore, in this
case the interval $I(H)$ is just the singleton $\{0\}$. Another
interesting example is the pendulum, with configuration space
$M=S^1$. The unstable equilibrium and the homoclinic points have
asymptotic Maslov index zero, as well as the non-contractible
periodic orbits. On the other hand, the contractible periodic orbits
have positive asymptotic Maslov index, growing from zero - for the
contractible periodic orbits which are close to the homoclinic orbits -
to a positive maximum - achieved at the stable equilibrium point. 
This follows from the fact that for such orbits the vertical subspace 
makes exactly one complete turn in one period, and from the fact that
the period increases with the amplitude. Therefore,
in this case $I(H)$ is the bounded interval $[0,a]$, where $a>0$ is the
asymptotic Maslov index of the stable equilibrium.

In both these examples the fundamental group of $M$ is infinite. 
Manifolds with finite
fundamental group exhibit a completely different behavior: for a
fairly general class of Hamiltonians $H$ the interval  $I(H)$ contains
the half-line $[0,+\infty)$.

We denote by $Y:T^*M \rightarrow TT^*M$ the Liouville vector field on
$T^*M$, whose expression in local coordinates is
\[
Y(q,p) = \sum_{j=1}^n p^j \partial_{p^j}.
\]

\begin{thm}
\label{dense}
Assume that the compact manifold $M$ has finite fundamental group, and
that the Hamiltonian $H\in C^{\infty}(\T\times T^*M)$ satisfies:
\begin{enumerate}
\item The action integrand $DH(t,x)[Y(x)] - H(t,x)$ is coercive on
  $\T \times T^*M$.
\item The Hamiltonian $H(t,x)$ is superlinear on the fibers of $T^*M$,
  uniformly in $t\in \T$.
\item The Hamiltonian vector field $X_H$ generates a global flow.
\end{enumerate}
Then $I(H)$ contains the half-line $[0,+\infty)$. More precisely,
the set of asymptotic Maslov indices assumed by the contractible
closed orbits with integer period is dense in $[0,+\infty)$.
\end{thm}

The function $DH[Y]-H$ is called action integrand because it
coincides with the integrand of the Hamiltonian action along
orbits. Condition (i) is a sort of radial convexity 
assumption for large values of $|p|$, and it is satisfied for instance 
by a Hamiltonian which is the Fenchel transform of a Tonelli
Lagrangian (see the next section). Condition (i), together with
(iii), implies a priori estimates for orbits with bounded action.
The superlinearity assumption (ii) guarantees that the periodic orbit
problem ``sees'' the whole topology of the free loop space of $M$. A
Hamiltonian $H$ with linear growth would typically have periodic
orbits with bounded Conley-Zehnder index, and it would produce
a bounded interval $I(H)$.

\medskip 

\begin{proof}
Assume first that $M$ is simply connected. Then the free loop space of
$M$ has infinitely many non-vanishing Betti numbers, as proved by
Sullivan in \cite{sul75}. So there is a
diverging sequence of natural numbers $(h_j)$ such that the
$h_j$-th singular homology group of the free loop space of $M$ does
not vanish.

Let $r \in [0,+\infty)$ and fix some $\epsilon>0$. The fact that
the sequence $(h_j)$ diverges easily implies that for every $k_0>0$
the set
\[
\set{\frac{h_j}{k}}{j\in \N, \; k\in \N, \; k\geq k_0}
\]
is dense in $[0,+\infty)$.
Therefore, we can find $j\in \N$ and $k\in \N$, $k>0$, such that
\[
\left| r - \frac{h_j}{k} \right| < \frac{\epsilon}{3}, \quad
\frac{2n}{k} < \frac{\epsilon}{3}.
\]

Since the $h_j$-th singular homology group of the free loop space of
$M$ does not vanish, and the Hamiltonian $H$ satisfies the above
conditions (i), (ii), and (iii), the Hamiltonian vector field $X_H$
has at least one
$k$-periodic orbit $x$ with Conley-Zehnder index $\mu_{CZ}(x)$ satisfying
\[
|\mu_{CZ}^k(x) - h_j| \leq 2n.
\]
This existence statement is proved by minimax arguments on the
space of finite energy solutions of the Floer equation on the cylinder,
see \cite[Theorem 6.3]{af07}. Let $\eta$ be the probability measure on
$\T \times T^*M$ associated to the $k$-periodic orbit $x$. By (\ref{bott2}),
\[
|\mu_{CZ}^k(x) - k \hat{\mu}(\eta) | \leq 2n,
\]
hence we conclude that
\begin{eqnarray*}
|\hat{\mu} (\eta) - r| \leq \frac{1}{k} |k \hat{\mu} (\eta) -
 \mu_{CZ}^k(x)| + \frac{1}{k} | \mu_{CZ}^k(x) - h_j | +
 \left|\frac{h_j}{k} - r \right| \\ \leq \frac{2n}{k} +
 \frac{2n}{k} + \frac{\epsilon}{3} < \epsilon,
\end{eqnarray*}
proving the density statement for simply connected configuration
spaces. If the fundamental group of $M$ is finite, the universal covering of
$M$ is compact. Therefore, the asymptotic Maslov indices of the orbits
with integer period of the Hamiltonian vector field lifted to such a
compact manifold are dense in $[0,+\infty)$. Projecting a periodic
orbit of the lifted system onto $M$ produces a contractible
periodic orbit with the same asymptotic Maslov index.
\end{proof} \qed

\section{Convex Hamiltonians}
\label{sect:convex}

Assume that the Hamiltonian $H$ satisfies the Tonelli conditions:
it is $C^2$-strictly convex and superlinear on the fibers of $T^*M$,
meaning that
\[
\partial_{pp} H(t,q,p) >0, \quad \lim_{|p| \rightarrow +\infty}
\frac{H(t,q,p)}{|p|} = +\infty,
\]
where the norm $|\cdot|$ comes from some Riemannian metric on $M$ (by the
compactness of $\T \times M$, the superlinearity condition does not
depend on the choice of this metric).
Then the Fenchel transform defines a smooth Lagrangian on $\T \times
TM$,
\[
L(t,q,v) := \max_{p\in T_q^*M} \left( p[v] - H(t,q,p) \right),
\]
which satisfies the Tonelli conditions, that is it is $C^2$-strictly
convex and superlinear on the fibers of $TM$. Conversely, a Tonelli
Lagrangian defines by the dual Fenchel transform a Tonelli Hamiltonian.

The Legendre duality defines a loop of diffeomorphisms from the tangent
bundle to the cotangent bundle,
\[
\mathcal{L} : \T \times TM \rightarrow \T \times T^* M, \quad
(t,q,v) \mapsto (t,q,D_v L(t,q,v)),
\]
such that
\[
L(t,q,v) = p[v] - H(t,q,p) \quad \iff \quad (t,q,p) =
\mathcal{L}(t,q,v).
\]
A curve $x(t)=(q(t),p(t))$ in the cotangent bundle is an orbit of the
Hamiltonian vector field $X_H$ if and only if the curve $q:\R \times
M$ is a solution of the Euler-Lagrange equation
\begin{equation}
\label{eulag}
\frac{d}{dt} \partial_v L (t,q(t),\dot{q}(t)) = \partial_q
L (t,q(t),\dot{q}(t)),
\end{equation}
if and only if for every $s<t$ the curve $q|_{[s,t]}$ is an extremal of the
Lagrangian action
\[
\mathbb{A}_s^t(\gamma) = \int_s^t
L(\tau,\gamma(\tau),\dot{\gamma}(\tau))\, d\tau,
\]
among all absolutely continuous curves $\gamma$ with end-points
$\gamma(s)=q(s)$ and $\gamma(t)=q(t)$.  The
second variation of $\mathbb{A}_s^t$ at the extremal $q$ is a
continuous symmetric bilinear form on the Hilbert space
$W$ consisting of the $W^{1,2}$ sections $\xi$ of $q^*(TM)$
such that $\xi(s)=0$, $\xi(t)=0$. The $C^2$-strict convexity of $L$
guarantees that this bilinear form is a compact perturbation of a
coercive form. Therefore, its Morse index $i_s^t(q)$ and its nullity
$\nu_s^t(q)$ are finite. By the classical relationship between the Morse
index and the Maslov index, we have the identities
\[
\underline{\mu}_t(s,x_0) = i_s^{s+t}(q), \quad
\overline{\mu}_t(s,x_0)  = i_s^{s+t}(q) + \nu_s^{s+t}(q), \quad
\forall s\in [0,1),\; t\in [0,+\infty),
\]
where $x_0$ is the point on $T^*M$ whose orbit starting at time $s$
corresponds to the
curve $t\mapsto (\gamma(t),\dot\gamma(t))$ via the Legendre transform.
See for instance \cite{dui76}. Therefore, in the case of convex
Hamiltonians the asymptotic Maslov index can be interpreted as an
asymptotic Morse index,
\[
\hat{\mu}(s,x_0) = \lim_{t\rightarrow +\infty} \frac{i_s^{s+t}(q)}{t} =
\lim_{t\rightarrow +\infty} \frac{i_s^{s+t}(q) + \nu_s^{s+t}(q)}{t},
\]
and this number is always non-negative.
In this context, a corollary of Theorem \ref{dense} is that if $L$ is
a Tonelli Lagrangian generating a global flow on
the tangent bundle of a compact manifold with finite fundamental
group, the asymptotic Morse indices of closed orbits with integer
period are dense in $[0,+\infty)$. In this case, the range of the
asymptotic Morse index is exactly $[0,+\infty)$.

Now we restrict our attention to electro-magnetic systems, produced by
Lagrangians of the form
\begin{equation}
\label{lag}
L(t,q,v)= \frac{1}{2}\langle v,v\rangle + \langle A(t,q),v\rangle
- V(t,q),
\end{equation}
where $\langle \cdot,\cdot\rangle$ is a Riemannian structure on $M$ -
the kinetic energy - $A$ is a smooth vector field - the magnetic
potential - and $V$ is a smooth real function - the scalar potential.
Both $A$ and $V$ are 1-periodic in time. In this case, the Legendre
transform is
\begin{equation}
\label{legendre}
\mathcal{L}(t,q,v) = (t,q,p), \quad\mbox{where } p[w] = \langle v + A(t,q),
w \rangle, \quad \forall w\in T_q M,
\end{equation}
and the Hamiltonian is
\[
H(t,q,p) = \frac{1}{2} \langle p - \alpha(t,q) , p - \alpha(t,q)
\rangle + V(t,q),
\]
where $\langle \cdot,\cdot\rangle$ denotes the inner product induced
on $T^*M$ and $\alpha$ is the 1-form given by
\[
\alpha(t,q)[v] =  \langle A(t,q),v\rangle, \quad \forall v\in T_q M.
\]
The corresponding Euler-Lagrange equation can be written in the form
\[
\nabla_t \dot\gamma + \partial_t A(t,\gamma) - (\nabla
A(t,\gamma)^* - \nabla A(t,\gamma)) \dot\gamma + \grad V(t,\gamma) = 0,
\]
where $\nabla$ is the covariant derivative associated to the metric of $M$,
$\nabla_t$ denotes the covariant derivative along $\gamma$, the
superscript $*$ denotes the metric adjoint, and $\grad$ denotes the
gradient. Therefore, if $\gamma$ is a solution of the Euler-Lagrange
equation, we have the estimate
\[
\frac{d}{dt} |\dot\gamma(t)|^2 = 2 \langle \nabla_t \dot\gamma(t),
\dot\gamma(t) \rangle \leq c_1 (1+|\dot\gamma(t)|^2)
\]
for some constant $c_1$. It follows that the Euler-Lagrange flow on
$\T \times TM$ - hence also the Hamiltonian flow on $\T \times T^*M$ -
is globally defined, and
\begin{equation}
\label{cresc}
1 + |\dot\gamma(t)|^2 \leq (1+ |\dot\gamma(s)|^2) e^{c_1 |t-s|}, \quad
 \forall s,t\in \R.
\end{equation}
The second variation of the Lagrangian action functional at an
extremal curve $\gamma$ on the time interval $[s,s+t]$ is
\begin{eqnarray*}
d^2 \mathbb{A}_s^{s+t}(\gamma) [\xi,\xi] = \int_s^{s+t} \Bigl( \langle
\nabla_{\tau} \xi,\nabla_{\tau} \xi \rangle + 2 \langle \nabla_{\xi}
A, \nabla_{\tau} \xi \rangle \\ + \langle \hess A [\xi,\xi] , \dot\gamma
\rangle - \hess V [\xi,\xi] - \langle R(\xi,\dot\gamma)\xi, \dot\gamma
+ A \rangle \Bigr)\, d\tau,
\end{eqnarray*}
where $\hess$ denotes the Riemannian Hessian and $R$ is the Riemann
tensor. The above formula shows that there exists a number $c_2$ such that
\[
d^2 \mathbb{A}_s^{s+t}(\gamma) [\xi,\xi] \geq \int_s^{s+t} \left(
\frac{1}{2}|\nabla_{\tau}
\xi|^2 - c_2 (1+ |\dot\gamma|^2) |\xi|^2 \right) \, d\tau.
\]
Together with (\ref{cresc}), this estimate implies that
\[
d^2 \mathbb{A}_s^{s+t}(\gamma) [\xi,\xi] \geq \frac{1}{2} \int_s^{s+t}
|\nabla_{\tau} \xi|^2\, d\tau  -
c_2  (1+ |\dot\gamma(s)|^2) e^{c_1 |t|} \int_s^{s+t} |\xi|^2  \, d\tau.
\]
Writing $\xi$ as $\xi = \sum_{j=1}^n \varphi_j \xi_j$, where
$\xi_1,\dots,\xi_n$ is an orthonormal frame along $\gamma$ built
by parallel transport, and $\varphi_1,\dots,\varphi_n$ are real-valued
functions, the quadratic form on the right-hand side equals
\[
\sum_{j=1}^n \left[ \frac{1}{2} \int_s^{s+t} |\dot\varphi_j|^2\, d\tau  -
c_2  (1+ |\dot\gamma(s)|^2) e^{c_1 |t|} \int_s^{s+t} |\varphi_j|^2  \, d\tau
\right].
\]
The Morse index of the above quadratic form on the Sobolev space
$W^{1,2}_0([s,s+t],\R^n)$ is
\[
n \left\lfloor \frac{|t|\sqrt{2 c_2  (1+ |\dot\gamma(s)|^2)} e^{c_1
    |t|/2}}{\pi} \right\rfloor.
\]
Therefore, the Morse index of a solution on a fixed time interval
grows at most linearly with the initial velocity: there exists a
positive function $t \mapsto c_3(t)$ such that
\[
i_s^{s+t}(\gamma) \leq c_3(t) (1+ |\dot\gamma(s)|).
\]
By the form (\ref{legendre}) of the Legendre transform, we deduce a
similar estimate for the function $\mu_t$:
\begin{equation}
\label{finfirst}
\mu_t(s,q_0,p_0) \leq c_4(t) (1+ |p_0|),
\end{equation}
for a suitable function $t \mapsto c_4(t)$.
In particular, if a $\phi$-invariant probability measure $\eta$ on
$\T \times T^*M$ has finite first moment, then the function $\mu_t$ is
in $L^1(\T \times T^*M,\eta)$, and the asymptotic Maslov index of
$\eta$ is well-defined.

\section{Electro-magnetic Lagrangians}

The aim of this section is to prove the following theorem:

\begin{thm}
\label{main}
Assume that the manifold $M$ is compact and has finite fundamental group.
Let $L\in C^{\infty}(\T
\times TM)$ be a Lagrangian of the form
(\ref{lag}). Let $r$ be a non-negative number. Then at least one of the
following statements holds:
\begin{enumerate}
\item There exists a contractible closed orbit with integer period and
asymptotic Maslov index $r$.
\item There exists a sequence of contractible closed orbits with
  integer minimal periods
  $k_j \rightarrow +\infty$ and asymptotic Maslov indices $\hat{\mu}_j
  \rightarrow r$. The corresponding sequence of probability
  measures on $\T \times TM$ narrowly converges to an invariant probability
  measure $\eta\in \mathscr{P}(\T \times TM)$ with finite second
  moment and asymptotic Maslov index $\hat{\mu}(\eta)=r$.
\end{enumerate}
In both cases, there is an invariant probability measure $\eta$ on $\T \times TM$ with finite second moment, asymptotic Maslov index $r$, and action bound
\begin{equation}
\label{actbou}
\mathbb{A}(\eta) := \int L(t,q,v) \, d\eta(t,q,v) \leq  C (1+r^2),
\end{equation}
where $C$ is a positive number.
\end{thm}

When the Lagrangian $L$ does not depend on time, the above statement
can be made more precise: either there is a contractible closed orbit
with asymptotic Maslov index $r$, or there is a sequence of
contractible closed orbits with minimal periods $T_j\rightarrow
+\infty$ such that the corresponding probability measures on $TM$ narrowly
converge to an invariant probability measure which is supported on a
(compact) energy level and has asymptotic Maslov index $r$. See
Remark \ref{final} after the end of the proof.

\medskip

Let $W^{1,2}(\R/k\Z, M)$ be the space of $k$-periodic
curves in $M$ of Sobolev class $W^{1,2}$. This is an infinite
dimensional Hilbert manifold, and it has the homotopy type of
$\Lambda(M)$, the free loop space of $M$. By the form (\ref{lag}) of
the Lagrangian, the (average) action functional
\[
\mathbb{A}_k(\gamma) := \frac{1}{k}
\int_0^k L(t,\gamma(t),\dot{\gamma}(t))\, dt
\]
is twice continuously differentiable on $W^{1,2}(\R/k\Z,
M)$. Its critical points are precisely the $k$-periodic solutions of
the Euler-Lagrange equation (\ref{eulag}), and the second
differential of $\mathbb{A}_k$ at a critical point $\gamma$ is a
compact perturbation of a coercive symmetric bilinear form on
$T_{\gamma} W^{1,2}(\R/k\Z, M)$. In particular, the Morse index
$m_k(\gamma)$ is finite. We denote by $m^*_k(\gamma)$ the large Morse
index of the critical point $\gamma$, that is the Morse index plus the
nullity. Since the elements of the kernel of the second differential
of $\mathbb{A}_k$ solve a second order ODE, there holds
\[
m_k(\gamma) \leq m_k^*(\gamma) \leq m_k(\gamma) + 2n.
\]
Furthermore, if $x$ is the $k$-periodic orbit on $T^*M$ corresponding
to $\gamma$ by the Legendre transform,
\[
m_k(\gamma) \leq \mu_{CZ}^k(x) \leq  m_k^*(\gamma),
\]
see \cite{dui76} and \cite{web02}.
Therefore, the asymptotic Maslov index of the invariant probability measure
$\eta$ associated to the $k$-periodic solution $\gamma$ can be
expressed in terms of the Morse indices of the Lagrangian action
functional by
\[
\hat{\mu}(\eta) = \lim_{h\rightarrow +\infty}
\frac{m_{hk}(\gamma)}{hk} = \lim_{h\rightarrow +\infty}
\frac{m_{hk}^*(\gamma)}{hk}.
\]
If we endow $W^{1,2}(\R/k\Z, M)$ with the complete
Riemannian metric
\[
\langle \xi,\eta \rangle_{W^{1,2}} = \int_0^k \left( \langle \xi,\eta
\rangle + \langle \nabla_t \xi, \nabla_t \eta \rangle \right)\, dt,
\quad \forall \xi,\eta\in T_{\gamma} W^{1,2}(\R/k\Z, M),
\]
the functional $\mathbb{A}_k$ satisfies the Palais-Smale condition on
$W^{1,2}(\R/k\Z, M)$ (see \cite{ben86}, or the appendix in
\cite{af07}). This fact allows to find $k$-periodic orbits by minimax
arguments. Indeed, let $h$ be a natural number such that the $h$-th singular
homology group of $\Lambda(M)$ - hence also of $W^{1,2}(\R/k\Z, M)$ -
does not vanish. Let $\Gamma_k(h)$ be the class of compact subsets $K$
of $W^{1,2}(\R/k\Z, M)$ such that the inclusion mapping $i:K
\hookrightarrow W^{1,2}(\R/k\Z, M)$ induces a non-zero homomorphism
between the $h$-th homology groups. Then
\begin{equation}
\label{minimax}
a_k(h) := \inf_{K \in \Gamma_k(h)} \max_{\gamma \in K} \mathbb A_k(\gamma)
\end{equation}
is finite, and there exists a critical
point $\gamma\in W^{1,2}(\R/k\Z, M)$ such that
\[
\mathbb{A}_k(\gamma) = a_k(h), \quad
m_k(\gamma) \leq h \leq m_k^*(\gamma).
\]
See for instance \cite{vit88}. In the following lemma we prove an
estimate on the growth rate of $a_k(h)$ in terms of $k$.

\begin{lem}
\label{lem1}
There exists a constant $c_1$ such that for every $h\in \N$ for which
$H_h(\Lambda(M))\neq 0$, and for every integer $k\geq 1$, there holds
\[
a_k(h) \leq \frac{2}{k^2} a_1(h) + c_1.
\]
\end{lem}

\begin{proof}
The reparametrization mapping
\[
\varphi_k:W^{1,2}(\R/\Z,M) \to W^{1,2}(\R/k\Z,M), \quad
\varphi_k(\gamma)(t) := \gamma(t/k),
\]
is a homeomorphism. Therefore,
\begin{equation}
\label{minimax2}
a_k(h) = \inf_{K \in \Gamma_k(h)} \max_{\gamma \in K} \mathbb A_k(\gamma)
=\inf_{K \in \Gamma_1(h)} \max_{\gamma \in K} \mathbb
A_k(\varphi_k(\gamma)).
\end{equation}
Let $\gamma \in W^{1,2}(\R/\Z,M)$. By a change of variable and by
simple manipulations
\begin{align*}
\mathbb{A}_k (\varphi_k(\gamma)) & = \frac{1}{k} \int_0^k
\Bigl[ \frac{1}{2k^2} \scal{\dot{\gamma}(t/k),\dot{\gamma}(t/k)}
+ \frac{1}{k}\scal{A(t,\gamma(t/k)),\dot{\gamma}(t/k)}
- V(t,\gamma(t/k)) \Bigr]\,dt \\
& = \frac{1}{k} \int_0^1 \Bigl[ \frac{1}{2k} \scal{\dot{\gamma}(t),
\dot{\gamma}(t)}
+ \scal{A(kt,\gamma(t)),\dot{\gamma}(t)} - k V(kt,\gamma(t))
\Bigr]\,dt\\
& = \frac{2}{k^2} \mathbb{A}_1(\gamma) + \int_0^1 \Bigl[ \frac{1}{k}
  \scal{A(kt,\gamma(t)),\dot\gamma(t)} - \frac{2}{k^2}
  \scal{A(t,\gamma(t)) , \dot\gamma(t)} \\ & - V(kt,\gamma(t)) +
  \frac{2}{k^2} V(t,\gamma(t)) - \frac{1}{2k^2} \scal{\dot\gamma(t),
  \dot\gamma(t)}\Bigr]\,dt.
\end{align*}
Together with the Cauchy-Schwarz inequality this yields the estimate
\begin{equation}
\label{stm}
\mathbb{A}_k (\varphi_k(\gamma)) \leq \frac{2}{k^2}
\mathbb{A}_1(\gamma) + \left(\frac{1}{k} + \frac{2}{k^2}\right)
\|A\|_{\infty} \|\dot\gamma\|_{L^2} + \left(1 + \frac{2}{k^2}\right)
\|V\|_{\infty} - \frac{1}{2k^2} \|\dot\gamma\|_{L^2}^2.
\end{equation}
Since $k\geq 1$,
\begin{eqnarray*}
\left(\frac{1}{k} + \frac{2}{k^2}\right)
\|A\|_{\infty} \|\dot\gamma\|_{L^2} - \frac{1}{2k^2} \|\dot\gamma\|_{L^2}^2
\leq \frac{3}{k} \|A\|_{\infty} \|\dot\gamma\|_{L^2} - \frac{1}{2k^2}
\|\dot\gamma\|_{L^2}^2 \\
\leq \max_{r\geq 0} \left(3 \|A\|_{\infty} r -
\frac{1}{2} r^2 \right) = \frac{9}{2} \|A\|_{\infty}^2.
\end{eqnarray*}
Therefore, (\ref{stm}) shows that for every $\gamma \in
W^{1,2}(\R/\Z,M)$ there holds
\[
\mathbb{A}_k (\varphi_k(\gamma)) \leq \frac{2}{k^2}
\mathbb{A}_1(\gamma) + c_1,
\]
where
\[
c_1 := 3 \|V\|_{\infty} + \frac{9}{2} \|A\|_{\infty}^2.
\]
The conclusion follows by (\ref{minimax2}).
\end{proof} \qed

In the next lemma we prove an estimate on the growth rate of $a_1(h)$
in terms of $h$.

\begin{lem}
\label{lem2}
If $M$ is compact and simply connected, then there exist
positive constants $c_2$ and $c_3$ such that
\[
a_1(h) \leq c_2 h^2 + c_3,
\]
for every $h\in \N$ for which $H_h(\Lambda(M))\neq 0$.
\end{lem}

\begin{proof}
Let $L(\gamma)$ be the length of a loop $\gamma$ with respect to the metric $\langle \cdot, \cdot \rangle$,
\[
L(\gamma) = \int_0^1 |\dot\gamma(t)| dt.
\]
Gromov has proved that the fact that the compact Riemannian manifold $M$ is simply connected has the following consequence: there exists a number $C_0$ such that every singular homology class $\alpha\in H_h(W^{1,2}(\R/\Z,M))$ is represented by a singular cycle whose support is contained in the length sublevel
\[
\set{\gamma \in W^{1,2}(\R/\Z,M) }{L(\gamma) \leq C_0 h},
\]
for every $h\in \N$ (see \cite{gro78} and \cite{gro99}, Theorem 7.3). By a reparametrization argument, we can transform this cycle into a cycle in the same homology class $\alpha$ consisting of loops $\gamma$ which satisfy the estimate
\[
\|\dot{\gamma}\|_{\infty} \leq C h,
\]   
for a suitable constant $C$, not depending on $h$. See Appendix A for details.
Therefore, for every $h\in \N$ for which the $h$-th homology group of the free loop space of $M$ does not vanish there exists $K$ in $\Gamma_1(h)$ such that
\[
\max_{\gamma\in K} \|\dot{\gamma}\|_{\infty} \leq C h.
\]
Then, if $\gamma$ belongs to $K$,
\begin{multline*}
\mathbb{A}_1(\gamma) \leq \int_0^1 \left( \frac{1}{2} |\dot\gamma|^2 + \|A\|_{\infty} |\dot\gamma| + \|V\|_{\infty} \right) \, dt \\ \leq \frac{1}{2} C^2 h^2 + \|A\|_{\infty} C h + \|V\|_{\infty} \leq C^2 h^2 + \frac{1}{2} \|A\|_{\infty}^2 + \|V\|_{\infty}.
\end{multline*}
We conclude that the desired estimate holds with $c_1=C^2$ and $c_2 = \|A\|_{\infty}^2/2 + \|V\|_{\infty}$. 
\end{proof} \qed

\begin{proof}[of Theorem \ref{main}]
By lifting the system to the universal covering, we may assume that
$M$ is simply connected. Up to adding a positive constant, we may
assume that the Lagrangian $L$ is non-negative.

Assume that (i) does not hold: there are
no contractible closed orbits with integer period and asymptotic
Maslov index $r$. Arguing as in the proof of Theorem
\ref{dense}, we can find two diverging sequences
of natural numbers $(h_j)$ and $(k_j)$ such that the $h_j$-th homology group of
$\Lambda(M)$ does not vanish, $k_j\geq 1$, and $(h_j/k_j)$ converges
to $r$. By the minimax (\ref{minimax}), we can find a
$k_j$-periodic solution $\gamma_j$ with
\[
\mathbb{A}_{k_j} (\gamma_j) = a_{k_j}(h_j), \quad m_{k_j}(\gamma_j)
\leq h_j \leq m_{k_j}^*(\gamma_j).
\]
By Lemmas \ref{lem1} and
\ref{lem2},
\[
\mathbb{A}_{k_j} (\gamma_j) = a_{k_j}(h_j) \leq \frac{2}{k_j^2}
a_1(h_j) + c_1 \leq 2 c_2 \frac{h_j^2}{k_j^2} + 2c_3 \frac{1}{k_j^2} +
c_1,
\]
so the average action of $\gamma_j$ has a uniform bound of the form,
\begin{equation}
\label{bound}
\mathbb{A}_{k_j}(\gamma_j) \leq C (r^2 +1), \quad \forall j\in \N.
\end{equation}

Let $k_j'\in \N$ be the minimal period of the periodic solution
$\gamma_j$. If by contradiction $k_j'$ does not diverge, up to a subsequence
we may assume that $k_j\equiv k$ is constant. Then $(\gamma_j)$ is a
sequence of $k$-periodic solutions with bounded action. It easily
follows that $(\gamma_j)$ is bounded in $C^2(\R/\Z,M)$ (see for
instance \cite[Lemma 1.1]{af07}), so a subsequence of $(\gamma_j)$
converges in $C^1$ to a contractible $k$-periodic solution $\gamma$ which has
asymptotic Maslov index $r$ (for instance because the associated
invariant probability measures on $\T \times TM$ converge
narrowly). This contradiction shows that the minimal
periods diverge.

Let $\eta_j$ be the invariant probability measure on $\T \times TM$
associated to $\gamma_j$, that is
\[
\eta_j(A) = \frac{1}{k_j} |\set{t\in
  [0,k_j]}{(t,\gamma_j(t),\dot\gamma_j(t))\in A}|,
\]
for every Borel subset $A$ of $\T \times TM$. Fix some number $M>0$.
By the Tchebichev inequality - recall that
$L\geq 0$ - and by (\ref{bound}),
\begin{multline}
\label{smb} \eta_j(\set{(t,q,v)\in \T \times TM}{L(t,q,v)\geq M})
\leq \frac{1}{M} \int_{\T \times TM} L(t,q,v)\, d\eta_j(t,q,v) \\
= \frac{1}{M} \frac{1}{k_j} \int_0^{k_j}
L(t,\gamma_j(t),\dot\gamma_j(t))\, dt = \frac{1}{M}
\mathbb{A}_{k_j}(\gamma_j) \leq \frac{c}{M}. \end{multline} Since
the function $L$ is coercive on $\T \times TM$, we conclude that
the sequence $(\eta_j)$ is tight, so by the Prokhorov theorem up
to a subsequence it narrowly converges to some invariant
probability measure $\eta$ on $\T \times TM$. Actually, since $L$
grows quadratically in $v$, the bound (\ref{smb}) shows that the
sequence $(\eta_j)$ has uniformly integrable second moments, so
$\eta$ has finite second moment (see for instance \cite[section
5.1.1]{ags05}). In particular, by (\ref{finfirst}) for every
$t\geq 0$ the function $\mu_t$ is uniformly integrable with
respect to the sequence $(\eta_j)$, so by Theorem \ref{cont} the
invariant probability measure $\eta$ has asymptotic Maslov index
$r$. This concludes the proof of (ii). Since the action is
lower semi-continuous with respect to the narrow topology of
measures, the estimate (\ref{actbou}) follows from (\ref{bound})
(in case (ii), case (i) is similar).
\end{proof} \qed

\begin{rem}
\label{final}
As mentioned before, if the Lagrangian $L$ does not depend on time
then either there is a contractible closed orbit
with asymptotic Maslov index $r$, or there is a sequence of
contractible closed orbits with diverging minimal periods such that
the corresponding probability measures on $TM$ narrowly
converge to an invariant probability measure which is supported on an
energy level and has asymptotic Maslov index $r$.

Indeed, the above proof produces a sequence of contractible closed
orbits $(\gamma_j)$ with asymptotic Maslov index converging to
$r$ and bounded average action. The energy $H$ - read on $TM$ by
the Legendre transform - is an integral of the Euler-Lagrange flow, so
the bound on the average action implies that each curve
$(\gamma_j,\dot\gamma_j)$ is contained in an energy level
$\{H=H_j\}$, with $(H_j)$ a bounded sequence of real numbers. If the
minimal periods of $\gamma_j$ are bounded, by an easy limiting
argument we obtain a contractible closed orbit with asymptotic Maslov index
$r$. Another consequence is that the corresponding probability measures
on $TM$ narrowly converge - up to a subsequence - to a probability
measure which is also supported on an energy level.
\end{rem}

\section{Minimizing measures and Aubry-Mather theory}

In this section we develop an analogue of the ergodicity analysis in
Aubry-Mather theory (see for instance \cite[Chapter III]{man91} or
\cite{sib04} for an introduction to the subject) replacing the
rotation vector of a closed measure with the asymptotic Maslov
index of the measure. 

For sake of simplicity, we deal with an electromagnetic Lagrangian
of the form (\ref{lag}) on a manifold $M$ which is assumed to be
compact and to have finite fundamental group. See Remark \ref{finrem} below
for possible generalizations.

By Theorem \ref{main}, for any $r\in [0,+\infty)$ there is an invariant probability measure $\eta$ on $\T \times TM$ with asymptotic Maslov index $\hat{\mu}(\eta)=r$ and finite action
\[
\mathbb{A}(\eta) := \int_{\T \times TM} L(t,q,v)\, d\eta(t,q,v).
\]
Therefore, we can define a real function $\beta$ on $[0,+\infty)$ as
\[
\beta(r):=\inf_{\hat{\mu} (\eta)= r} \mathbb{A}(\eta),
\]
where the infimum is taken over all the invariant measures on $\T \times TM$ for which $\mu_t$ is integrable for every $t\geq 0$, and $\hat{\mu}(\eta)=r$.

\begin{prop}
\label{pbta}
For every $r\geq 0$, the infimum in the definition of $\beta(r)$ is attained. Moreover, the function $\beta$ has the following properties:
\begin{enumerate}
\item[(i)] $\beta$ is convex;
\item[(ii)] $\beta$ has quadratic growth at infinity, meaning that
\[
a_1 r^2 - A_1 \leq \beta(r) \leq a_2 r^2 + A_2,
\]
for some positive constants $a_1,a_2,A_1,A_2$.
\end{enumerate}
\end{prop}

\begin{proof}
In order to prove that the infimum defining $\beta(r)$ is
attained, it suffices to show that for every $r\geq 0$ and $c\in
\R$ the space of invariant measures with asymptotic Maslov index
$r$ and action not exceeding $c$ is compact in the narrow
topology. Let $(\eta_j)$ be a sequence of invariant measures with
$\hat{\mu}(\eta_j)=r$ and $\mathbb{A}(\eta_j)\leq c$. Since $L$
grows quadratically in $v$ at infinity, the upper bound on
$\mathbb{A}(\eta_j)$ implies that the sequence $(\eta_j)$ has
uniformly integrable second moments. In particular $(\eta_j)$ is
tight, so by Prokhorov theorem up to a subsequence it narrowly
converges to some probability measure $\eta$. The measure $\eta$
is invariant, and since $\mathbb{A}$ is narrowly lower
semi-continuous we get $\mathbb{A}(\eta)\leq c$. By the estimate
(\ref{finfirst}), for every $t\geq 0$ the function $\mu_t$ is
uniformly integrable with respect to the sequence of measures
$(\eta_j)$, so Theorem \ref{cont} implies that the invariant
probability measure $\eta$ has asymptotic Maslov index $r$. This
concludes the proof of the narrow compactness.

(i) Since the asymptotic Maslov index is linear in $\eta$, that is
\[
\hat{\mu}(s \eta_1 + (1-s) \eta_2)=s \hat{\mu}(\eta_1) + (1-s) \hat{\mu} (\eta_2)
\qquad \text{for }0 \leq s \leq 1,
\]
the convexity of $\beta$ follows easily from the definition.

(ii) The upper quadratic bound follows from the existence of an
invariant measure satisfying the estimate (\ref{actbou}) in
Theorem \ref{main}. As for the lower quadratic bound, notice that
by (\ref{unif}), by (\ref{finfirst}), and by the Cauchy-Schwarz
inequality, any invariant probability measure $\eta$ satisfies
\begin{multline*}
\hat{\mu} (\eta) \leq 2n + M_1 (\eta) = 2n + \int \mu_1(s,q,p)\, d\eta(s,q,p)  \\
\leq 2n+ c_4(1) \int (1+|p|) \, d\eta(s,q,p) \leq 2n + c_4(1) + c_4(1) \left( \int |p|^2\, d\eta(s,q,p) \right)^{1/2}.
\end{multline*}
Since $L$ has a lower quadratic bound, the above inequality implies an estimate of the form
\[
\hat{\mu}(\eta) \leq b_1 + b_2 \left( \int L(s,q,p) \, d\eta(s,q,p) \right)^{1/2} =
b_1 + b_2 \, \mathbb{A}(\eta)^{1/2},
\]
for suitable positive numbers $b_1,b_2$. The lower quadratic bound for $\beta$ follows.
\end{proof} \qed

Let us denote by $\mathcal M_r \subset
\set{\eta}{\hat{\mu}(\eta)=r}$ the set of invariant measures with
minimal action among those with asymptotic Maslov index $r$.
Thanks to the convexity and quadratic growth of $\beta$ we can now
prove that, although in general minimizing measures are not
ergodic, there exists a diverging sequence of values $r\in
[0,+\infty)$ for which $\mathcal M_r$ contains an ergodic measure.

We recall the concept of ergodic decomposition of an invariant
measure. Let $\eta$ be an invariant probability measure on $\T \times TM$,
and define $\mathscr{E}(\eta)$ to be the space of ergodic 
invariant probability measures whose support is contained in the
support of $\eta$. We endow this space with the narrow topology. Then 
there exists a unique probability measure $\lambda$ on the Borel 
$\sigma$-algebra of $\mathscr{E}(\eta)$ such that
\[
\eta(A) = \int_{\mathscr{E}(\eta)} \nu(A)\, d\lambda(\nu),
\]
for every Borel set $A\subset \T \times TM$. The measure $\lambda$ is 
called the \textit{ergodic decomposition} of $\eta$ and the ergodic 
measures $\nu \in \supp(\lambda)$ are called the \textit{ergodic
components} of $\eta$. Since $\eta$ need not have compact support,
this decomposition result does not follow at once from
Choquet theorem, but it requires a more sophisticated proof. It is
proved, in the more general setting of a locally compact group (in our
case $\R$) acting in a measurable way on a complete separable metric 
space (in our case the support of $\eta$), by Varadarajan in 
\cite[Theorem 4.4]{var63}.   

We also recall that a point $r \in [0,+\infty)$ is said to be an extremal point of the convex function $\beta$ if $\beta(r) < s\beta(r_1) + (1-s)\beta(r_2)$ for all $r_1,r_2 \in [0,+\infty)$ and $0 < s < 1$ satisfying $sr_1 + (1-s)r_2=r$.
We can now prove the following:

\begin{prop}
If $r$ is an extremal point of $\beta$, then $\mathcal M_r$ contains an ergodic measure.
\end{prop}

Combining this proposition with the fact that $\beta$ has
quadratic growth at infinity, we conclude that there exists a
diverging sequence $(r_j)$ such that $\mathcal M_{r_j}$ contains
an ergodic measure.

\begin{proof}
Let us fix $\eta \in \mathcal M_r$.
By the ergodic decomposition theorem stated above, we can write
\begin{equation}
\label{eqbetadecomp}
\beta(r)=\int_{\T \times TM} L \,d\eta=\int_{\mathscr{E}(\eta)} 
\left( \int_{\T \times TM} L \,d\nu\right) \,d\lambda(\nu).
\end{equation}
Moreover, by Theorem \ref{erg},
\[
\hat{\mu}(\eta)=\int_{\T \times TM} \hat{\mu}(s,x) \,d\eta=
\int_{\mathscr{E}(\eta)} \left( \int_{\T \times TM} \hat{\mu}(s,x)
\,d\nu\right) \,d\lambda(\nu) =\int_{\mathscr{E}(\eta)} 
\hat{\mu}(\nu) \,d\lambda(\nu).
\]
Since $\beta$ is convex, by Jensen inequality we get
\[
\beta(r)=\beta(\hat{\mu}(\eta)) = \beta \left(
\int_{\mathscr{E}(\eta)} \hat{\mu}(\nu)
\, d\lambda(\nu) \right) \leq \int_{\mathscr{E}(\eta)} 
\beta(\hat{\mu}(\nu)) \,d\lambda(\nu),
\]
which together with (\ref{eqbetadecomp}) gives
\[
\int_{\mathscr{E}(\eta)} \left( \int_{\T \times TM} 
L \,d\nu\right) \,d\lambda(\nu) \leq \int_{\mathscr{E}(\eta)} 
\beta(\hat{\mu}(\nu)) \,d\lambda(\nu).
\]
Since, by the definition of $\beta$, $\int_{\T \times TM} L \,d\nu
\geq \beta(\hat{\mu}(\nu))$, we conclude that
\[
\int_{\T \times TM} L \,d\nu=\beta(\hat{\mu}(\nu)) \quad \mbox{and} \quad \beta(\hat{\mu}(\eta))
= \int_{\mathscr{E}(\eta)} \beta(\hat{\mu}(\nu)) \,d\lambda(\nu)
\qquad \text{for }\lambda-\text{a.e. }\nu,
\]
The latter identity, combined with the fact that
$r=\hat{\mu}(\eta)$ is an extremal point of $\beta$, implies that $\hat{\mu}(\nu)=r$ for $\lambda$-a.e. $\nu$. We have shown that $\lambda$-a.e.\ ergodic component of an element of $\mathcal{M}_r$ belongs to $\mathcal{M}_r$.
\end{proof} \qed

\begin{rem}
We observe that, without the assumption that $r$ is an extremal
point of $\beta$, from the equality $\beta(\hat{\mu}(\eta)) =
\int_{\mathscr{E}(\eta)} 
\beta(\hat{\mu}(\nu)) \,d\lambda(\nu)$ one deduces that, for
$\lambda$-a.e. $\nu$, $\beta(\hat{\mu}(\nu))$ belongs to the so
called supported domain of $r$, that is the interval given by
\[
\bigcap_{p \in \partial \beta(r)}
\{u \in [0,+\infty) \mid \beta(u)= \beta(r)+p(u-r)\},
\]
where $\partial \beta(r)$ denotes the subdifferential of $\beta$ at $r$, that is
\[
\partial \beta(r):=\{p \in \R \mid \beta(u)= \beta(r)+p(u-r) \ \forall u \in [0,+\infty)\}.
\]
\end{rem}

\begin{rem}
\label{finrem}
The results of this section can be extended to an autonomous Tonelli 
Lagrangian (the assumption of non-dependence on $t$ allows to recover 
compactness in a simple way), proving in this case a superlinear
growth at infinity on the function $\beta$ (see Proposition
\ref{pbta} (ii)). The configuration space $M$ could be any compact 
manifold, provided that we replace the half-line $[0,+\infty)$ by the 
interval $I(H)$, the range of the asymptotic Maslov index.
\end{rem}

\renewcommand{\thesection}{\Alph{section}}
\setcounter{section}{0}
\renewcommand{\theequation}{\Alph{section}\arabic{equation}}
\setcounter{equation}{0}

\section{Appendix: Singular cycles with bounded speed in the loop space}

Let $(M, \langle\cdot,\cdot\rangle)$ be a Riemannian manifold, and let
\[
L(\gamma) = \int_0^1 |\dot\gamma(t)|\, dt
\]
denote the length of an absolutely continuous curve $\gamma: [0,1]
\rightarrow M$. 
Gromov \cite{gro78} has proved  the following result:

\begin{thm}
\label{length}
Assume that the manifold $M$ is compact and simply connected. Then
there is a number $C_0$ such that for every $h\in \N$ every singular homology class
\[
\alpha\in H_h(C^0(\R/\Z,M)) 
\]
is represented by a singular cycle whose support is contained in
the length sublevel
\[
\set{\gamma\in C^0(\R/\Z,M)}{\gamma \mbox{ is absolutely continuous and } L(\gamma) \leq C_0 h}.
\]
\end{thm}

See also Theorem 7.3 in \cite{gro99}, and Theorem 5.10 in the book of Paternain \cite{pat99} for detailed proofs. The cycle produced by Gromov's argument is made of loops whose speed may have large variation, so the length estimate does not imply a good estimate on the uniform norm of the speed, and not even on the energy (as erroneously claimed in \cite{pat99}). However, such an estimate can be easily obtained by a second homotopy, as shown by the following:   

\begin{cor}
\label{speed}
Assume that the manifold $M$ is compact and simply connected. Then
there exists a number $C$ such that for every $h\in \N$ every
singular homology class
\[
\alpha\in H_h(C^{\infty}(\R/\Z,M))
\]
is represented by a singular cycle whose support is contained in
\[
\set{\gamma\in C^{\infty}(\R/\Z,M)}{\|\dot{\gamma}\|_{\infty} \leq C h}.
\]
\end{cor}

Since the inclusions
\[
C^{\infty}(\R/\Z,M) \hookrightarrow W^{1,2}(\R/\Z,M) 
\hookrightarrow C^0(\R/\Z,M)
\]
are homotopy equivalences, the above result holds a fortiori when the
space of smooth loops is replaced by the space of $W^{1,2}$ loops, or
by the space of continuous loops. Corollary \ref{speed} can be deduced from the proof of Lemma 2.12 in \cite[Lemma 2.12]{fs06}. For sake of completeness, we include a proof also here.

\medskip

\begin{proof}
We may assume that $h\geq 1$, because the (unique) zero-homology class is realized by any constant loop. Let $\alpha$ be a homology class of degree $h$ and let $a_0$ be the singular cycle given by Theorem \ref{length}, so that
\begin{equation}
\label{stma}
[a_0] = \alpha, \quad L(\gamma) \leq C_0 h, \quad \forall \gamma\in \supp a_0.
\end{equation}
By regularization and up to choosing a larger $C_0$, we may assume that $a_0$ is a singular cycle in $C^{\infty}(\R/\Z,M)$. For instance, such a regularization can be achieved by embedding $M$ into some Euclidean space, by smoothing the loops in $\alpha$ by a convolution keeping them in a tubular neighborhood of $M$, and by projecting onto $M$.

In order to obtain a cycle $a_1$ which is homologous to $a_0$ and is
supported in a suitable speed sublevel, the natural idea would be to
build a homotopy changing the parametrization of the loops in $a_0$, so
that they become parametrized by constant speed. 
However, such a construction presents technical difficulties due to
the fact that the speed of some loops may vanish somewhere. We can overcome this difficulty by reparametrizing the
loops in $a_0$ by slowly varying speed, as follows.

Given a loop $\gamma\in C^{\infty}(\R/\Z,M)$, we define the real function $\sigma_{\gamma}$ on $[0,1]$ by
\[
\sigma_{\gamma}(t) := \frac{1}{\int_0^1 \sqrt{1+|\dot{\gamma}(s)|^2}\, ds} \int_0^t \sqrt{1+|\dot{\gamma}(s)|^2} \, ds.
\]
Notice that $\sigma_{\gamma}(0)=0$, $\sigma_{\gamma}(1)=1$. Since the square root is subadditive, we have
\begin{equation}
\label{der}
\sigma_{\gamma}' (t) = \frac{\sqrt{1+|\dot{\gamma}(t)|^2}}{\int_0^1 \sqrt{1+|\dot{\gamma}(s)|^2}\, ds} \geq \frac{\sqrt{1+|\dot{\gamma}(t)|^2}}{1+ L(\gamma)} >0,
\end{equation}
so $\sigma_{\gamma}$ is a smooth diffeomorphism of $[0,1]$ onto itself. Denote by $\tau_{\gamma}$ its inverse, and notice that the mapping $\gamma\mapsto \tau_{\gamma}$ is continuous in the $C^{\infty}$ topology. Then the homotopy
\[
F : [0,1] \times C^{\infty}(\R/\Z,M) \rightarrow C^{\infty}(\R/\Z,M), \quad F(\lambda,\gamma)(t) := \gamma(\lambda \tau_{\gamma}(t) + (1-\lambda)t),
\]
is continuous, so the cycle $a_0= F(0,\cdot)_* a_0$ is homologous to the cycle
\[
a_1 := F(1,\cdot)_* a_0.
\]
If $\gamma_1= F(1,\gamma)$, $\gamma$ in the support of $a_0$, is a loop in the support of $a_1$, differentiating the identity $\gamma_1(\sigma_{\gamma}(t)) = \gamma(t)$ we find by (\ref{der})
\[
|\dot{\gamma}_1(\sigma_{\gamma}(t))| = \frac{|\dot{\gamma}(t)|}{\sigma_{\gamma}^{\prime}(t)} \leq \frac{|\dot{\gamma}(t)|}{\sqrt{1+|\dot{\gamma}(t)|^2}} (1+L(\gamma)) \leq 1 + L(\gamma).
\]
Therefore, by (\ref{stma})
\[
\|\dot{\gamma}_1\|_{\infty} \leq 1+L(\gamma) \leq 1 + C_0 h \leq  (1+C_0) h.
\]
Therefore, $a_1$ is the desired cycle in $C^{\infty}(\R/\Z,M)$.
\end{proof}\qed

\end{document}